\documentclass[a4paper]{article}
\usepackage{amsmath}
\usepackage{amsfonts}
\usepackage{amssymb}
\usepackage{amsthm}
\usepackage{graphicx}
\usepackage{caption}
\usepackage{enumerate}
\usepackage{color}
\begin{document}

\newtheorem{thm}{Theorem}[section]
\newtheorem{cor}[thm]{Corollary}
\newtheorem{lem}[thm]{Lemma}
\newtheorem{prop}[thm]{Proposition}
\newtheorem{defn}[thm]{Definition}
\newtheorem{rem}[thm]{Remark}
\newtheorem{example}[thm]{Example}
\newtheorem{question}[thm]{Question}
\newtheorem{algorithm}[thm]{Algorithm}
\newtheorem{problem}[thm]{Problem}
\numberwithin{equation}{section}

\title{The finite basis problem for the monoid of $2 \times 2$ upper triangular tropical matrices
\thanks{Yuzhu Chen, Xun Hu and Yanfeng Luo are partially supported by Natural Science Foundation of China (No. 11371177, 11401275).}}
\author{Yuzhu Chen$^1$, Xun Hu$^{1, 2}$, Yanfeng Luo$^1$\thanks{Corresponding author} and Olga Sapir$^3$\\
\footnotesize\em 1.School of Mathematics and Statistics, Lanzhou University, \\
\footnotesize\em Lanzhou, Gansu, $730000$, P. R. of China\\
\footnotesize\em 2.Department of Mathematics and Statistics, Chongqing Technology and Business University,\\
\footnotesize\em Chongqing, $400033$, P. R. of China\\
\footnotesize\em 3. Department of Mathematics, Vanderbilt University, Nashville, TN37240, USA \\
\footnotesize\em email: olga.sapir@gmail.com\\}
\date{}

\maketitle

\begin{abstract}
For each positive $n$, let  ${\bf u}_n \approx {\bf v}_n$ denote the identity obtained from the Adjan identity $ (xy) (yx) (xy) (xy) (yx) \approx (xy) (yx) (yx) (xy) (yx)$ by substituting $(xy) \rightarrow (x_1 x_2 \dots x_n)$ and $(yx) \rightarrow (x_n \dots x_2 x_1)$. We show that every monoid which satisfies  ${\bf u}_n \approx {\bf v}_n$ for each positive $n$ and generates a variety containing the bicyclic monoid is nonfinitely based.

 This implies that the monoid $U_2(\mathbb{T})$ (resp., $U_2(\overline{\mathbb{Z}})$) of $2 \times 2$ upper triangular tropical matrices over the tropical semiring $\mathbb{T} = \mathbb{R} \cup \{ -\infty\}$ (resp., $\overline{\mathbb{Z}} =\mathbb{Z} \cup \{ -\infty\}$) is nonfinitely based.

\vskip 0.1in


\end{abstract}

\section{Introduction}

In the past years, tropical algebra (also known as max-plus algebra) as the linear algebra carried out over the tropical semiring has been intensively studied. In particular, the monoid and semiring of all $n\times n$ tropical matrices plays an important role both in theoretical algebraic study and in applications to combinatorics, geometry and semigroup representations, as well as to optimisation and scheduling problems (\cite{But}), formal language and automata theory (\cite{Sim}), control theory (\cite{Coh}) and statistical inference (\cite{Pach}).

 Adjan's identity
$xyyxxyxyyx\approx xyyxyxxyyx$
was introduced in \cite{Adj} by Adjan as the first known and the shortest nontrivial identity satisfied by the bicyclic monoid $\mathfrak B$.
Izhakian and Margolis \cite{Zur} studied the identities of the monoid $U_2(\mathbb{T})$ of all $2 \times 2$ tropical matrices over the tropical semiring $\mathbb{T}$ by the use of tropical algebra and proved that $U_2(\mathbb{T})$ satisfies the Adian's identity. They also proved that $U_2(\mathbb{T})$ contains a copy of the bicyclic monoid thus reproving Adjan's identity in a much simpler way than in \cite{Adj}.

An algebra $A$ is said to be \textit{finitely based} if the set $\mathsf{Id}(A)$
of all identities it satisfies can be derived from a finite subset of $\mathsf{Id}(A)$. Otherwise, it is said to
be \textit{nonfinitely based}. The finite basis problem asks if there is an algorithm to determine when an algebra
is finitely based. Although McKenzie \cite{McKenzie} proved that this problem is undecidable for general
algebras, the problem is still open for many classes of algebras. Since the end of the 1960s, the finite basis
problem for semigroups has been studied intensively (see the survey \cite{Volkov} and recent articles \cite{LiZhangLuo, luozhangja, OS1, zhangliluo2012, zhangliluo}), but  still remains open.

Schneerson \cite{Sch89} studied the identities of the bicyclic monoid and proved that this monoid has an infinite axiomatic rank.  Pastijn \cite{Pastin} described the identities of the bicyclic monoid $\mathfrak B$ in terms of systems of linear
inequalities and proved that every basis of identities for $\mathfrak B$ contains, for every $n \ge 3$, an infinity of identities involving precisely $n$ variables. Hence the bicyclic monoid is nonfinitely based.

Johnson and Kambites \cite{John} and Izhakian and Margolis (unpublished) explored the algebraic structure of $U_2(\mathbb{T})$ and characterized the Green's relations on it. Shitov \cite{Shi} determined the subgroups of the monoid $U_n(\mathbb{T})$ of all tropical $n \times n$ matrices. By using the correspondence between tropical matrices and weighted digraphs, Izhakian \cite{Zur1} proved that $U_n(\mathbb{T})$ satisfies a nontrivial identity and provided a generic construction for classes of such identities.

In 1968, Perkins~\cite{Perkins} established a sufficient condition under which a semigroup is nonfinitely based and used it to prove that the 6-element Brandt monoid $B_{2}^{1}$ is nonfinitely based.
Later, many other sufficient conditions for the nonfinite basis property of semigroups were established.
While most of these conditions are syntactic, some of them are not.
For example, the sufficient condition of Volkov~\cite{MV} which implies the nonfinite basis property of the 6-element semigroup $A_2^g$ is not syntactic.
While most syntactic sufficient conditions are similar to the original Perkins sufficient condition, some of them are not.
For example, the result of M. Sapir~\cite{MS} that a finite semigroup $S$
is inherently nonfinitely based  if and only if every Zimin word (${\bf Z}_1=x_1, \dots, {\bf Z}_{k+1} = {\bf Z}_kx_{k+1}{\bf Z}_k, \dots$)  is an isoterm for $S$
yields a syntactic sufficient condition which is not similar to the Perkins sufficient condition.

Zhang and Luo \cite{Zhang-Luo} proved that the 6-element semigroup $L$ is nonfinitely based
which gives the fourth and the last~\cite{LeeLiZhang} example of a minimal nonfinitely based semigroup.  Lee~\cite{lee} generalized the results of \cite{Zhang-Luo} to a sufficient condition for the nonfinite basis property of semigroups. Article \cite{OS} contains a general method for proving that a semigroup is nonfinitely based. This method works well for proving those sufficient conditions which are similar to the original Perkins sufficient condition. In particular, by using this method O. Sapir reduced the number of requirements in both Perkins' and Lee's sufficient conditions (see \cite[Section 5]{OS}). Recently, Lee modified his sufficient condition into an even weaker sufficient condition under which a semigroup is nonfinitely based (private communication).

Recall that there exist several powerful methods to attack the finite basis
problem for finite semigroups (see \cite{Volkov} for details). But, to the
best of our knowledge, so far the problem has been solved for only a few families of infinite semigroups.  Recently, Auinger et al \cite{Achlv} established a new sufficient condition under which a semigroup (finite or infinite) is nonfinitely based. As an application, it is shown that the Kauffman monoid $K_n$ and the wire monoid $W_n$ either as semigroups or as involution semigroups are nonfinitely based for each $n\geq 3$. This sufficient condition is proved by using the sufficient condition in \cite{MS} and is also different from the Perkins sufficient condition.

In this paper, we present a new sufficient condition (see Theorem \ref{thm:suff} below) under which a semigroup is nonfinitely based.
Let $\sim_S$ denote the fully invariant congruence on the free semigroup $\mathcal X^+$ corresponding to a semigroup $S$. Like all the other sufficient conditions similar to the original Perkins condition, Theorem \ref{thm:suff} exhibits a certain (finite) set of words $W$, a certain set of identities $\Sigma$ in unbounded number of variables and states the following:

$\bullet$ If a monoid $S$ satisfies all the identities in $\Sigma$ and the words in $W$ are $\sim_S$-related to other words in $\mathcal X^+$ in a certain way, then the monoid $S$ is nonfinitely based.

But unlike in most other sufficient conditions in the Perkins club, the set of words $W$ involved in our sufficient
condition contains some words with three non-linear (occurring more than once) variables.

For each positive $n$, let  ${\bf u}_n \approx {\bf v}_n$ denote the identity obtained from the Adjan identity $ (xy) (yx) (xy) (xy) (yx) \approx (xy) (yx) (yx) (xy) (yx)$ by substituting $(xy) \rightarrow (x_1 x_2 \dots x_n)$ and $(yx) \rightarrow (x_n \dots x_2 x_1)$.
Using the sufficient condition in Theorem \ref{thm:suff} we show that every monoid which satisfies  ${\bf u}_n \approx {\bf v}_n$ for each positive $n$ and generates a variety containing the bicyclic monoid $\mathfrak B$ is nonfinitely based (see Theorem \ref{thm:main} below).

We use the result in \cite{Zur1} to show that  $U_2(\mathbb{T})$ satisfies ${\bf u}_n \approx {\bf v}_n$ for each positive $n$.
Thus Theorem \ref{thm:main} and the result of Izhakian and Margolis imply that the monoid $U_2(\mathbb{T})$ (resp., $U_2(\overline{\mathbb{Z}})$) of $2 \times 2$ upper triangular tropical matrices over the tropical semiring $\mathbb{T} = \mathbb{R} \cup \{ -\infty\}$ (resp., $\overline{\mathbb{Z}} =\mathbb{Z} \cup \{ -\infty\}$) is nonfinitely based.

\section{Preliminaries}\label{sec:pre}
Most of the notations and background material used in this paper are given in this section. The reader is referred to \cite{Aki}, \cite{burris} and \cite{howie} for any undefined notation and terminology.

\subsection{Tropical matrices}

Tropical algebra is carried out over the \emph{tropical semiring} $\mathbb{T} = (\mathbb{R} \cup \{ -\infty\}, \oplus, \odot)$ (see, for example, \cite{Cun}), the set $\mathbb{R}$ of real numbers together with minus infinity $ -\infty $, with the addition and multiplication defined as follows
\begin{equation}\label{eq:trop}
  a\oplus b =\max\{ a,b\},~a\odot b = a+b.
\end{equation}
In other words, the \emph{tropical} \emph{sum} of two numbers is their maximum and the \emph{tropical} \emph{product} of two numbers is their sum. It is clear that both the addition and multiplication are commutative.
Furthermore, $\mathbb{T}$ is an additively idempotent semiring, i.e., $a\oplus a = a$ for any $a \in \mathbb{T}$, in which $ -\infty$ is the zero element and $0$ is the unit.

Let $M_n(\mathbb{T})$ be the semiring of all $ n \times n$ matrices with entries in the tropical semiring $\mathbb{T}$,  in which the addition and multiplication are induced from $\mathbb{T}$, as in the familiar matrix construction. It is easy to see that
\[
I =\begin{pmatrix}
    0 &  \cdots & -\infty \\
     \vdots &     \ddots & \vdots  \\
     -\infty &  \cdots   &  0 \\
  \end{pmatrix}_{n\times n} \; \mbox{and}\quad
 O = \begin{pmatrix}
    -\infty &  \cdots & -\infty \\
     \vdots &     \ddots & \vdots  \\
     -\infty &  \cdots   & -\infty \\
  \end{pmatrix}_{n\times n}
\]
are the  unit element and the zero element of $M_{n}(\mathbb{T})$, respectively. In particular, $M_{n}(\mathbb{T})$ is a monoid with respect to its multiplication, and in this paper it is always considered as a monoid.
The submonoid of all upper (resp., lower) triangular tropical matrices is denoted by $U_n(\mathbb{T})$ (resp., $L_n(\mathbb{T}))$.

Let $A=(a_{ij}), B=(b_{ij}) \in M_n(\mathbb{T}) $. $A$ and $B$ are said to be \emph{diagonally equivalent} if $a_{ii} = b_{ii}$ for each $i = 1,2,\ldots ,n$, written as $A \sim_{diag} B$.

\subsection{Semigroup identities}
Let $\mathcal{X}$ be a countably infinite alphabet and let $\mathcal{X}^+$ and $\mathcal{X}^* = \mathcal{X}^+ \cup \{1\}$  be the free semigroup and the free monoid over $\mathcal{X}$ respectively, where $1$ is the empty word. Elements of $\mathcal{X}$ are called \textit{letters} or \textit{variables} and elements of $\mathcal{X}^{*}$ are called \textit{words}.

In this paper, $a, b, c, \ldots, x, y, z$ with or without indices stand for letters and $\textbf{a}, \textbf{b}, \textbf{c}$, $\ldots, \textbf{x}, \textbf{y}, \textbf{z}$  with or without indices stand for words.

Let $x$ be a letter and $\mathbf{w}$ be a word. Then
\begin{itemize}

\item the \textit{content} of $\mathbf{w}$, denoted by
$\textsf{con}(\mathbf{w})$, is the set of all different letters occurring in
$\mathbf{w}$;

\item $\textsf{occ}(x,\mathbf{w})$ is the number of occurrences of the letter $x$ in
$\mathbf{w}$;

\item the \textit{length} of a word $\mathbf{w}$, denoted by $|\mathbf{w}|$, is the number of (not necessarily distinct) letters appearing in $\mathbf{w}$, i.e., $|\mathbf{w}|=\sum_{x\in \textsf{con}(\mathbf{w})}\textsf{occ}(x,\mathbf{w})$;

\end{itemize}

An \textit{identity} is a formal expression $\mathbf{u}\approx \mathbf{v}$ where $\mathbf{u}, \mathbf{v} $
are nonempty words.
We write $\textbf{u} = \textbf{v}$ if $\textbf{u}$ and $\textbf{v}$ are identical words.
We say an identity $\textbf{u}\approx \textbf{v}$ is $\emph{non-trivial}$ if $\textbf{u}\not = \textbf{v}$.  Let $S$ be a semigroup.
An identity $\mathbf{u}\approx \mathbf{v}$ is said to be satisfied by $S$ (written $S\vDash \mathbf{u}\approx
\mathbf{v}$) if the equality $\varphi(\mathbf{u})=\varphi(\mathbf{v})$ holds in $S$ for all possible homomorphisms $\varphi: \mathcal{X}^+\rightarrow S$.
Such a homomorphism is called an $\textit{assignment}$.
We say that $S$ satisfies a set of identities $\Sigma$ (written $S \vDash \Sigma$) if it satisfies every identity in $\Sigma$.
A $\textit{substitution}$ $\theta$ is a semigroup homomorphism $ \theta: \mathcal{X}^+\rightarrow \mathcal{X}^+$ defined by its action on $\mathcal{X} $.

Denote by $\mathsf{Id}(S)$ the set of all identities satisfied by $S$.
Given an identity system $\Sigma$, we denote by $\mathsf{Id}(\Sigma)$ the set of all consequences of $\Sigma$.
An $\textit{identity basis}$ for a semigroup $S$ is any set $\Sigma \subseteq \mathsf{Id}(S)$ such that
$\mathsf{Id}(\Sigma)=\mathsf{Id}(S)$, that is, every identity satisfied by $S$ can be derived from $\Sigma$.
A semigroup $S$ is called $\textit{finitely based}$ if it possesses a finite identity basis, otherwise, $S$ is
said to be $\textit{nonfinitely based}$.

Let $\textbf{w}$ be a word and $\mathcal A = \{x_{1}, x_{2}, \ldots, x_{m}  \} $ be a set of variables. We denote by $\textbf{w}(\mathcal A)$ or $\textbf{w}(x_{1}, \ldots, x_{m} )$ the word obtained from $\textbf{w}$ by deleting every occurrence of the variables in $\mathsf{con}(\textbf{w}) \backslash \mathcal A$.
In this case, we say that the word $\textbf{w}$ $\textit{deletes}$ to the word $\textbf{w}(\mathcal A)$. Note that if a semigroup $M$ is a monoid, that is, contains an identity element, then for any set of variables $\mathcal A$ we have $M \models \textbf{u}(\mathcal A)\approx \textbf{v}(\mathcal A)$ whenever $M \models \textbf{u} \approx \textbf{v}$.

A word $\textbf{u}$ is called an $\textit{isoterm}$ for a semigroup $S$, if
$S \models \textbf{u} \approx \textbf{v}$ if and only if $\textbf{u} = \textbf{v}$.  Note that
if $\textbf{u}$ is an isoterm for a semigroup $S$, then so are all nonempty subwords of $\textbf{u}$.
We say that a set of variables $\mathcal A \subseteq \mathcal X$ is {\em stable} in an identity $\bf u \approx \bf v$ if ${\bf u}(\mathcal A)={\bf v}(\mathcal A)$. Otherwise, we say that the set $\mathcal A$ is {\em unstable} in $\bf u \approx \bf v$.  We say that a set of variables $\mathcal A$ is {\em stable} in a word $\bf u$ with respect to a semigroup $S$ if the set $\mathcal A$ is stable in every identity of $S$ of the form $\bf u \approx \bf v$.
Two variables $x$ and $y$ are said to be {\em adjacent} in a word $\bf u$ if some occurrences of $x$ and $y$ are adjacent in $\bf u$.

\begin{lem} \cite[Fact 3.4 ($(i) \leftrightarrow (v)$)]{OS} \label{fact1} For a monoid $S$ and a word ${\bf u}$
the following conditions are equivalent:

(i) ${\bf u}$ is an isoterm for $S$;

(ii) Each adjacent pair of variables in ${\bf u}$ is stable in ${\bf u}$ with respect to $S$.

\end{lem}

\begin{lem} \cite[Fact 3.5(ii)]{OS} \label{fact2} If a set of variables $\mathcal A$ is stable in an identity ${\bf u} \approx {\bf v}$, then
every subset of $\mathcal A$ is also stable in ${\bf u} \approx {\bf v}$.
\end{lem}

\section{A sufficient condition under which a semigroup is nonfinitely based}\label{sec: condition}

We say that a word $\bf u$ is {\em applicable} to $\bf U$ if $\Theta({\bf u}) = {\bf U}$ for some substitution $\Theta: \mathcal{X} \rightarrow \mathcal X^+$.

\begin{lem} \cite[Corollary 2.2]{OS} \label{ml} Let $S$ be a semigroup. Suppose that for each $n$ large enough one can find a word ${\bf U}_n$ in at least $n$ variables such that ${\bf U}_n$ is not an isoterm for $S$ but every word $\bf u$ in less than $n/2$ variables applicable to ${\bf U}_n$ is an isoterm for $S$. Then $S$ is nonfinitely based.
\end{lem}

  As in \cite{OS}, given a substitution $\Theta: \mathcal X \rightarrow \mathcal X^+$ and a set of variables $\mathcal Y \subseteq \mathcal X$, we define $\Theta^{-1}(\mathcal Y):= \{x \in \mathcal X \mid \textsf{con}(\Theta(x)) \cap \mathcal Y \ne \emptyset \}$.
The following theorem gives a sufficient condition under which a semigroup is nonfinitely based.

\begin{thm} \label{thm:suff}  Let S be a monoid satisfying the following conditions:

(i) any word in more than one variable of length five is an isoterm for $S$;

(ii) any word in more than two variables applicable to $(xy) (yx) (xy) (xy) (yx)$ is an isoterm for $S$;

(iii) the word $xyz^{i_1} yz^{i_2} xz^{i_3} xyxyz_1^{i_1} yz_1^{i_2} xz_1^{i_3}$ is an isoterm for $S$, where
$z$ and $z_1$ are possibly equal and
$i_1 +i_2 + i_3 = 1$;

(iv) for any positive integer $n$, S satisfies the identity

$(x_1 \dots x_n)(x_n \dots x_1)(x_1 \dots x_n)(x_1 \dots x_n)(x_n \dots x_1)$

$\approx (x_1 \dots x_n)(x_n \dots x_1)(x_n \dots x_1)(x_1 \dots x_n)(x_n \dots x_1).$

Then S is nonfinitely based.
\end{thm}

\begin{proof} Fix $n$ large enough.
By the assumption, the word
\[
{\bf U}_n = (x_1 \dots x_n)(x_n \dots x_1)(x_1 \dots x_n)(x_1 \dots x_n)(x_n \dots x_1)
\]
is not an isoterm for $S$. Let $\bf u$ be a word in less than $n/2$ variables such that for some substitution $\Theta: \mathcal X \rightarrow \mathcal X ^+$ we have $\Theta({\bf u}) = {\bf U}_n$.
If $|\textsf{con}({\bf u})| = 1$ then ${\bf u} =x$ and therefore, is an isoterm for $S$ by Condition (i). So, we may assume that the word $\bf u$ depends on at least two variables.

In view of Lemma \ref{fact1}, in order to prove that the word ${\bf u}$ is an isoterm for $S$, it is enough to verify that each adjacent pair of distinct variables in $\textsf{con}({\bf u})$ is stable in $\bf u$ with respect to $S$. Since each adjacent pair of variables in $\textsf{con}({\bf u})$ forms a subset of $\Theta^{-1} (\{p,q\})$ for some adjacent pair $\{p,q\}\subset \textsf{con}({\bf U}_n)$, it is enough to verify that for each adjacent pair $\{p,q\}\subset \textsf{con}({\bf U}_n)$ the set $\Theta^{-1} (\{p,q\})$ is stable in $\bf u$ with respect to $S$ whenever the set $\Theta^{-1} (\{p, q\})$ contains at least two variables  (see Lemma \ref{fact2}).

If $p=q$ then the set $\Theta^{-1} (\{p\})$ is stable in $\bf u$ with respect to $S$ because of Condition (i). Now we assume that $p \ne q$. If the set $\Theta^{-1} (\{p, q\})$ contains more than two variables, then it is stable in $\bf u$ with respect to $S$ by Condition (ii).
Now we assume that the set $\Theta^{-1} (\{p, q\})$ contains exactly two variables $x$ and $y$. If $|{\bf u}(x,y)| <10$ then modulo renaming variables ${\bf u}(x,y) \in \{xy, xyx, xyxxy \}$. Since each of these words is an isoterm for $S$ by Condition (i),
the set $\Theta^{-1} (\{p, q\}) = \{x,y\}$ is stable in $\bf u$ with respect to $S$. If $|{\bf u}(x,y)|= 10$ then without loss of generality we may assume that $\Theta(x)=p$ and $\Theta(y)=q$. Consider three cases.

{\bf Case 1}. $\{p,q\} = \{x_i, x_{i+1} \}$ for some $1 \le i \le n/2$.

Since the word $\bf u$ has less than $n/2$ variables, for some letter $z \in \textsf{con}({\bf u})$, $\Theta(z)$ contains the subword  $x_{(j+1)}x_j$ for some $j>n/2$.
Since the subword $x_{(j+1)}x_j$ occurs only twice in ${\bf U}_n$, the word $\bf u$ deletes to some word
$(xy) z (yx) (xy) (xy) z' (yx)$ where $z'$ is possibly equal to $z$. Since by Condition (iii) this word is an isoterm for $S$, the pair $\{x,y\}$ is stable in $\bf u$ with respect to $S$.

{\bf Case 2}. $\{p,q\} = \{x_i, x_{i+1} \}$ for some $n/2 < i < n$.

Since the word $\bf u$ has less than $n/2$ variables, for some letter $z \in
\textsf{con}({\bf u})$, $\Theta(z)$ contains the subword  $x_{(j+1)}x_j$ for some $j \le n/2$.
Since the subword $x_{(j+1)}x_j$ occurs only twice in ${\bf U}_n$, the word $\bf u$ deletes to some word
$(xy) (yx) z (xy) (xy) (yx) z'$ where $z'$ is possibly equal to $z$. Since by Condition (iii) this word is an isoterm for $S$, the pair $\{x,y\}$ is stable in $\bf u$ with respect to $S$.

{\bf Case 3}. $\{p,q\} = \{x_1, x_n \}$.

Since the word $\bf u$ has less than $n/2$ variables, for some letter $z \in \textsf{con}({\bf u})$, $\Theta(z)$ contains the subword  $x_{(j+1)}x_j$ for some $1 < j < n-1$.
Since the subword $x_{(j+1)}x_j$ occurs only twice in ${\bf U}_n$, the word $\bf u$ deletes to some word
$(xy) y z x  (xy) (xy) y z' x$ where $z'$ is possibly equal to $z$. Since by Condition (iii) this word is an isoterm for $S$, the pair $\{x,y\}$ is stable in $\bf u$ with respect to $S$.

Therefore, the monoid $S$ is nonfinitely based by Lemma \ref{ml}.
\end{proof}

\section{Some properties of the identities of the bicyclic monoid $\mathfrak B$} \label{sec:bmon}

The monoid $\mathfrak B = \langle A, B \rangle$, generated by two elements $A$ and $B$ satisfying the relation
\[
AB =1,
\]
where 1 is the identity element, is called the {\em bicyclic monoid}.

Recall that Adjan's identity
\[
xyyxxyxyyx\approx xyyxyxxyyx
\]
was introduced in \cite{Adj} by Adjan as the first known and the shortest nontrivial identity satisfied by the bicyclic monoid. Hence we have
the following.

\begin{lem}\label{isot9}
Any word of length less than $10$ is an isoterm for $\mathfrak B$.
\end{lem}

The next result gives some prohibited identities for $\mathfrak B$.

\begin{lem}\label{lem:bcns}
The bicyclic monoid $\mathfrak B$ does not satisfy the following identities
\[
xy z^{i_1} yz^{i_2}xz^{i_3}  xy xy z^{i_1}yz^{i_2}xz^{i_3}\approx xy z^{i_1} yz^{i_2}xz^{i_3}  yx xy z^{i_1}yz^{i_2}xz^{i_3}
\]
with $ i_1+i_2+i_3=1$.
\end{lem}

\begin{proof}
Let $\varphi: \mathcal{X}^+ \rightarrow \mathfrak B$ be the assignment defined by
\[
t \mapsto \left\{
            \begin{array}{cl}
            A^2,& \mbox{ if }  t = x,\\
            B^3,& \mbox{ if }  t = y,\\
            A^3,&  \text{otherwise} .
            \end{array}\right.
\]
Then
\[
\varphi(xyzyxxyxyzyx)=BA^2\neq B^2A^3=\varphi(xyzyxyxxyzyx).
\]

Let $\varphi: \mathcal{X}^+ \rightarrow \mathfrak B$ be the assignment defined by
\[
t \mapsto \left\{
            \begin{array}{cl}
            B^2,& \mbox{ if }  t = x,\\
            A^3,& \mbox{ if }  t = y,\\
            B^3,&  \text{otherwise}.
            \end{array}\right.
\]
Then
\[
\varphi(xyyzxxyxyyzx ) =\varphi(xyyxzxyxyyxz )=B^3A^2
\]
and
\[
\varphi(xyyzxyxxyyzx)= \varphi(xyyxz yxxyyxz)=B^2A.
\]

Therefore, $\mathfrak B$ does not satisfy
\[
xy z^{i_1} yz^{i_2}xz^{i_3}  xy xy z^{i_1}yz^{i_2}xz^{i_3}\approx xy z^{i_1} yz^{i_2}xz^{i_3}  yx xy z^{i_1}yz^{i_2}xz^{i_3}.
\]
\end{proof}

By a FORTRAN program, Shleifer \cite{Shl90} proved that Adjan's identity and identity
\[
xyyxxyyxxy\approx xyyxyxyxxy
\]
are the only two identities in the alphabet $\{x, y\}$ of length $10$ satisfied by the bicyclic monoid. Thus we have

\begin{lem}\label{lem:iso55}
If $\mathfrak B \vDash xyyxxyxyyx \approx \textbf{v}$, then $\textbf{v}=xyyxxyxyyx$ or $xyyxyxxyyx$.
\end{lem}

\section{The monoid of $2 \times 2$ upper triangular tropical matrices is nonfinitely based}  \label{sec:upper}

For each positive integer $n$, let
\begin{equation*}\label{eq:suff}
\begin{array}{ccc}
\hspace{-0.7in}\textbf{u}_{n}\hspace{-0.1in} & =\hspace{-0.1in} & (x_1 \cdots  x_n)(x_n\cdots  x_1)(x_1 \cdots  x_n)(x_1 \cdots  x_n)(x_n \cdots  x_1), \vspace{0.05in}\\
\hspace{-0.7in}\textbf{v}_{n}\hspace{-0.1in} & =\hspace{-0.1in} & (x_1 \cdots  x_n)(x_n\cdots   x_1)(x_n\cdots   x_1)(x_1\cdots  x_n)(x_n\cdots  x_1).
\end{array}
\end{equation*}
Set $\Sigma=\{\textbf{u}_{n}\approx\textbf{v}_{n}|\ n\in \mathbb{N}\}$. Denote by $[\Sigma]$ the semigroup variety determined by $\Sigma$.

\begin{thm}\label{thm:main}
Every monoid $M$ such that $\mathfrak{B}\in \mathsf{var}~M \subseteq [\Sigma]$ is nonfinitely based.
\end{thm}

\begin{proof} Let $M$ be a monoid such that $\mathfrak B \in \mathsf{var}~M \subseteq [\Sigma]$.
In order to show that $M$ is nonfinitely based it is enough to verify that the bicyclic monoid satisfies
the conditions (i)-(iii) of Theorem \ref{thm:suff}.

 If $|\mathbf{u}| = 5$ then $\mathbf{u}$
 is an isoterm for $\mathfrak B$ by Lemma \ref{isot9}. That is, $\mathfrak B$ satisfies the condition (i) of Theorem 3.1.

 Let $\mathbf{u}$ be any word applicable to $(xy)(yx)(xy)(xy)(yx)$. If $|\mathbf{u}| < 10$ then $\mathbf{u}$ is an isoterm for $\mathfrak B$ by Lemma \ref{isot9}.
 If $|\mathbf{u}| = 10$ and $|\mathsf{con}(\mathbf{u})| >2$  then $|\mathbf{u}(z_1, z_2)|<10$ for any $z_1, z_2\in \mathsf{con}(\mathbf{u})$. It follows from Lemma \ref{isot9} that $\mathbf{u}(z_1, z_2)$ is an isoterm for $\mathfrak B$. Therefore, by Lemma \ref{fact1} the word $\mathbf{u}$ is also an isoterm for $\mathfrak B$. That is, $\mathfrak B$ satisfies the condition (ii) of Theorem 3.1.

 Let
\[
\mathfrak B \vDash \mathbf{u}=xyz^{i_1}yz^{i_2}xz^{i_3}xyxyz_1^{i_1}yz_1^{i_2}xz_1^{i_3} \approx \mathbf{v}
 \]
for some word $\mathbf{v}$ and $i_1 +i_2+i_3=1$, where $z$ and $z_1$ are possibly equal. Since $\mathfrak B$ satisfies the identity $\mathbf{u}(x, y)\approx\mathbf{v}(x, y)$, from Lemma \ref{lem:iso55} we have
\[
 \mathbf{v}(x,y)\in \{xyyxxyxyyx, xyyxyxxyyx\}.
\]
Note that $|\mathbf{u}(x,z,z_1)|=|\mathbf{u}(y,z,z_1)|<10$. It follows from Lemma \ref{isot9} that $\mathbf{v}(x,z,z_1)=\mathbf{u}(x,z,z_1)$ and $\mathbf{v}(y,z,z_1)=\mathbf{u}(y,z,z_1)$. Therefore, either $ \mathbf{v}=\mathbf{u}$ or
\[
\mathbf{v}=xy z^{i_1} yz^{i_2}xz^{i_3}  yx xy z_1^{i_1}yz_1^{i_2}xz_1^{i_3}
\]
with $i_1+i_2+i_3=1$. Now from Lemma \ref{lem:bcns} we must have $\mathbf{v}=\mathbf{u}$. That is, $\mathfrak B$ satisfies the condition (iii) of Theorem 3.1.
\end{proof}

Notice that the proof of Theorem \ref{thm:main} yields a short and natural explanation of why the bicyclic monoid $\mathfrak B$ is nonfinitely based \cite{Pastin, Sch89}.

In order to prove that $U_2(\mathbb{T}) \models \Sigma$ we use the following result from \cite{Zur1}.

\begin{lem}\cite[Theorem 4.2]{Zur1}\label{lem:first}
Any two matrices $A,B \in U_2(\mathbb{T})$ such that $A \sim_{\text{diag}} B$ satisfy the identity
\[ AB A AB = AB B AB.\]
\end{lem}

\begin{lem}\label{lem:uidentity}
Let $\textbf{u}, \textbf{v} \in \mathcal{X}^+$ such that $\textsf{occ}(x, \textbf{u}) = \textsf{occ}(x, \textbf{v})$ for any $x\in \textsf{con}(\textbf{uv})$. Then
\[
 U_2(\mathbb{T}) \vDash \textbf{u}\textbf{v}\textbf{u} \textbf{u}\textbf{v}\approx \textbf{u}\textbf{v} \textbf{v} \textbf{u}\textbf{v}.
\]
\end{lem}
\begin{proof} Note that $AB\sim_{\text{diag}} BA$ for any $A,B \in U_2(\mathbb{T})$.
Since $\textsf{occ}(x, \textbf{u}) = \textsf{occ}(x, \textbf{v})$ for any $x\in \textsf{con}(\textbf{uv})$, we have $ \varphi(\textbf{u}) \sim_{diag} \varphi(\textbf{v})$ for any assignment $\varphi: \mathcal{X}^+ \rightarrow U_2(\mathbb{T})$. Now the lemma follows from Lemma~\ref{lem:first} immediately.
\end{proof}

\begin{cor}\label{ut_2}
The monoid $U_2(\mathbb{T})$ of $2 \times 2$ upper triangular tropical matrices is nonfinitely based.
\end{cor}

\begin{proof} Lemma \ref{lem:uidentity} implies immediately that
for any positive integer $n $, $U_2(\mathbb{T})$ satisfies the identity
\begin{equation*}\label{eq:iden}
\begin{split}
     &{\bf u}_n = (x_1 \cdots  x_n)(x_n\cdots x_1)(x_1 \cdots  x_n)(x_1\cdots  x_n)(x_n\cdots x_1) \\
    \approx& (x_1 \cdots  x_n)(x_n\cdots x_1)(x_n\cdots x_1)(x_1 \cdots  x_n)(x_n\cdots x_1) = {\bf v}_n.
\end{split}
\end{equation*}


Let $\mathfrak B$ be the submonoid of $U_2(\mathbb{T})$ generated by the two elements
\[
A= \left(
     \begin{array}{cc}
       -1 & 1 \\
       -\infty & 1\\
     \end{array}
   \right)\ \makebox{and}\ B= \left(
     \begin{array}{cc}
        1 & 1 \\
       -\infty & -1\\
     \end{array}
   \right).
\]
It is proved in \cite{Zur} that $\mathfrak B$ is a bicyclic monoid.
Therefore, the monoid $U_2(\mathbb{T})$ is nonfinitely based by Theorem \ref{thm:main}.
\end{proof}

Let $\overline{\mathbb{Z}} = (\mathbb{Z}\cup \{ -\infty \}; \oplus, \odot )$ be the tropical semiring over $\mathbb{Z}\cup\{-\infty\}$, in which the addition $\oplus $ and multiplication $\odot$ are defined by~\eqref{eq:trop}. Then the monoid $U_2(\overline{\mathbb{Z}})$ of $2 \times 2$ upper triangular matrices over $\overline{\mathbb{Z}}$ is a submonoid of $ U_2(\mathbb{T})$ and $\mathfrak B$ is a submonoid of $U_2(\overline{\mathbb{Z}})$. It follows from Theorem~\ref{thm:main} that

\begin{cor}\label{thm:UZ}
The monoid $U_2(\overline{\mathbb{Z}})$ is nonfinitely based.
\end{cor}

Since for each positive $n$ the identity $\textbf{u}_{n}\approx\textbf{v}_{n}$ has $n$ variables, we get

\begin{cor}
$U_2(\mathbb{T})$ and $U_2(\overline{\mathbb{Z}})$ are both of infinite axiomatic rank.
\end{cor}

\noindent{\bf Acknowledgements} Yuzhu Chen, Xun Hu and Yanfeng Luo would like to thank Dr. Jianrong Li for his helpful discussion. The authors would also like to thank Dr. Gili Golan for finding a hole in the proof of Theorem \ref{thm:suff}.

\end{document}